\documentclass{article}
\usepackage{amsfonts}
\usepackage{amsmath}

\newtheorem{theorem}{Theorem}

\newtheorem{definition}[theorem]{Definition}

\newtheorem{lemma}[theorem]{Lemma}
\newtheorem{notation}[theorem]{Notation}

\newtheorem{proposition}[theorem]{Proposition}

\newenvironment{proof}[1][Proof]{\noindent\textbf{#1.} }{\ \rule{0.5em}{0.5em}}
\input{tcilatex}
\begin{document}

\title{The Lipschitz continuity of the solution to branched rough
differential equations}
\author{Jing Zou, Danyu Yang\thanks{%
danyuyang@cqu.edu.cn}}
\maketitle

\begin{abstract}
Based on an isomorphism between Grossman Larson Hopf algebra and Tensor Hopf
algebra, we apply a sub-Riemannian geometry technique to branched rough
differential equations and obtain the explicit Lipschitz continuity of the
solution with respect to the initial value, the vector field and the driving
rough path.
\end{abstract}

\section{Introduction}

In his seminal paper \cite{lyons1998differential}, Lyons built the theory of
rough paths. The theory gives a meaning to differential equations driven by
highly oscillating signals and proves the existence, uniqueness and
stability of the solution to differential equations. The theory has an
embedded component in stochastic analysis, and has been successfully applied
to differential equations driven by general stochastic processes \cite%
{lyons2002system, chevyrev2019canonical, friz2008uniformly,
friz2010differential}, the existence and smoothness of the density of
solutions \cite{cass2010densities, cass2015smoothness}, stochastic Taylor
expansions \cite{friz2008euler}, support theorem \cite{chevyrev2018support},
large deviations theory \cite{ledoux2002large} etc.

In Lyons' original framework \cite{lyons1998differential}, highly
oscillating paths are lifted to geometric rough paths in a nilpotent Lie
group. Geometric rough paths take values in a truncated group of characters
of the shuffle Hopf algebra \cite[Section 1.4]{reutenauer1993free} and
satisfy an abstract integration by parts formula. Limits of continuous
bounded variation paths in a rough path metric are geometric. For example,
Brownian sample paths enhanced with Stratonovich iterated integrals are
geometric rough paths. However, the geometric assumption can sometimes be
restrictive. It\^{o} iterated integrals do not satisfy the integration by
parts formula and It\^{o} Brownian rough paths are not geometric. Moreover,
non-geometric rough paths appear naturally when solving stochastic partial
differential equations \cite{gubinelli2010ramification}.

To provide a natural framework for non-geometric rough paths, Gubinelli \cite%
{gubinelli2010ramification} introduced branched rough paths and proved the
existence, uniqueness and continuity of the solution to branched rough
differential equations. Branched rough paths take values in a truncated
group of characters of Connes Kreimer Hopf algebra \cite{connes1998hopf}.
The multiplication of Connes Kreimer Hopf algebra is the free abelian
multiplication of monomials of trees which does not impose the integration
by parts formula. Branched rough paths can accomodate non-geometric
stochastic integrals and Connes Kreimer Hopf algebra provides a natural
algebraic setting for stochastic partial differential equations \cite%
{gubinelli2010ramification, hairer2014theory, bruned2019algebraic}.

The stability of the solution to rough differential equations is a central
result in rough path theory, commonly referred to as the Universal Limit
Theorem \cite[Theorem 5.3]{lyons2007differential}. Based on the uniform
decay of the differences between adjacent Picard iterations, Lyons \cite[%
Theorem 4.1.1]{lyons1998differential} proved the uniform continuity of the
solution with respect to the driving geometric rough path. Through
controlled paths \cite{gubinelli2004controlling}, Gubinelli \cite[Theorem 8.8%
]{gubinelli2010ramification} proved the Lipschitz continuity of the solution
to branched rough differential equations with respect to the initial value
and the driving rough path. Following the controlled paths approach, Friz
and Zhang \cite[Theorem 4.20]{friz2018differential} proved the Lipschitz
continuity of the solution to differential equations driven by branched
rough paths with jumps. Based on Davie's discrete approximation method \cite%
{davie2008differential} and by employing a sub-Riemannian geometry technique 
\cite{friz2008euler}, Friz and Victoir \cite[Theorem 10.26]%
{friz2010multidimensional} proved the explicit Lipschitz continuity of the
solution to differential equations driven by weak geometric rough paths over 
$%
%TCIMACRO{\U{211d} }%
%BeginExpansion
\mathbb{R}
%EndExpansion
^{d}$ with respect to the initial value, the vector field and the driving
rough path.

In this paper, we will extend Friz and Victoir's approach and result \cite[%
Theorem 10.26]{friz2010multidimensional} to branched rough differential
equations. Classically, the sub-Riemannian geometry technique only applies
to geometric rough paths. Based on an isomorphism between Grossman Larson
Hopf algebra and Tensor Hopf algebra \cite{foissy2002finite,
chapoton2010free}, Boedihardjo and Chevyrev \cite{boedihardjo2019isomorphism}
proved that branched rough paths are isomorphic to a class of $\Pi $-rough
paths \cite{gyurko2008numerical, gyurko2016differential}. A $\Pi $-rough
path \cite{gyurko2008numerical, gyurko2016differential} is an inhomogeneous
geometric rough path, for which the regularities of the components of the
underlying path are not necessarily the same. By applying a sub-Riemannian
geometry technique to $\Pi $-rough paths, we prove in Theorem \ref{Theorem
main theorem} the explicit Lipschitz dependence of the solution to branched
rough differential equations.

Comparing with the current existing results \cite[Theorem 8.8]%
{gubinelli2010ramification}\cite[Theorem 4.20]{friz2018differential}, our
result only requires that the vector field is $Lip\left( \gamma \right) $
for $\gamma >p$ (instead of $Lip\left( \left[ p\right] +1\right) $) and
explicitly specifies the uniform Lipschitz continuity of the solution with
respect to the initial value, the vector field and the driving branched
rough path, with the constant only depending on $p,\gamma ,d$ (the roughness
of the driving branched rough path, the regularity of the vector field and
the dimension of the underlying driving path).

\section{Notations}

A rooted tree is a finite connected graph with no cycle and a special vertex
called root. We call a rooted tree a tree. We assume trees are non-planar
for which the children trees of each vertex are commutative. A forest is a
commutative monomial of trees. The degree $\left\vert \tau \right\vert $ of
a forest $\tau $ is given by the number of vertices in $\tau $.

For the label set $\mathcal{L}:=\left\{ 1,2,\dots ,d\right\} $, an $\mathcal{%
L}$-labeled forest is a forest for which each vertex is attached with a
label from $\mathcal{L}$. Let $\mathcal{T}_{\mathcal{L}}$($\mathcal{F}_{%
\mathcal{L}}$) denote the set of $\mathcal{L}$-labeled trees (forests). Let $%
\mathcal{T}_{\mathcal{L}}^{N}$($\mathcal{F}_{\mathcal{L}}^{N}$) denote the
subset of $\mathcal{T}_{\mathcal{L}}$($\mathcal{F}_{\mathcal{L}}$) of degree 
$1,2,\dots ,N$.

Let $G_{\mathcal{L}}^{N}$ denote the group of degree-$N$ characters of $%
\mathcal{L}$-labeled Connes Kreimer Hopf algebra \cite[p.214]{connes1998hopf}%
. $a\in G_{\mathcal{L}}^{N}$ iff $a:%
%TCIMACRO{\U{211d} }%
%BeginExpansion
\mathbb{R}
%EndExpansion
\mathcal{F}_{\mathcal{L}}^{N}\rightarrow 
%TCIMACRO{\U{211d} }%
%BeginExpansion
\mathbb{R}
%EndExpansion
$ is an $%
%TCIMACRO{\U{211d} }%
%BeginExpansion
\mathbb{R}
%EndExpansion
$-linear map that satisfies%
\begin{equation*}
\left( a,\tau _{1}\tau _{2}\right) =\left( a,\tau _{1}\right) \left( a,\tau
_{2}\right)
\end{equation*}%
for every $\tau _{1},\tau _{2}\in \mathcal{F}_{\mathcal{L}}^{N}$, $%
\left\vert \tau _{1}\right\vert +\left\vert \tau _{2}\right\vert \leq N$,
where $\tau _{1}\tau _{2}$ denotes the commutative multiplication of
monomials of trees. The multiplication in $G_{\mathcal{L}}^{N}$ is induced
by the coproduct of Connes Kreimer Hopf algebra based on admissible cuts 
\cite[p.215]{connes1998hopf}: for $a,b\in G_{\mathcal{L}}^{N}$ and $\tau \in 
\mathcal{F}_{\mathcal{L}}^{N}$,%
\begin{equation*}
\left( ab,\tau \right) :=\sum_{\left( \tau \right) }\left( a,\tau _{\left(
1\right) }\right) \left( b,\tau _{\left( 2\right) }\right) .
\end{equation*}%
We equip $a\in G_{\mathcal{L}}^{N}$ with the norm:%
\begin{equation*}
\left\Vert a\right\Vert :=\max_{\tau \in \mathcal{F}_{\mathcal{L}%
}^{N}}\left\vert \left( a,\tau \right) \right\vert ^{\frac{1}{\left\vert
\tau \right\vert }}.
\end{equation*}

\begin{definition}[$p$-variation]
For a topological group $\left( G,\left\Vert \cdot \right\Vert \right) $,
suppose $X:\left[ 0,T\right] \rightarrow \left( G,\left\Vert \cdot
\right\Vert \right) $ is continuous. For $0\leq s\leq t\leq T$, denote%
\begin{equation*}
X_{s,t}:=X_{s}^{-1}X_{t}.
\end{equation*}%
For $p\geq 1$, define the $p$-variation of $X$ on $\left[ 0,T\right] $ as%
\begin{equation*}
\left\Vert X\right\Vert _{p-var,\left[ 0,T\right] }:=\sup_{D\subset \left[
0,T\right] }\left( \sum_{k,t_{k}\in D}\left\Vert
X_{t_{k},t_{k+1}}\right\Vert ^{p}\right) ^{\frac{1}{p}},
\end{equation*}%
where the supremum is taken over $D=\left\{ t_{k}\right\} _{k=0}^{n}$, $%
0=t_{0}<t_{1}<\cdots <t_{n}=T$, $n\geq 1$. Denote the set of continuous
paths from $\left[ 0,T\right] $ to $G$ with finite $p$-variation as $%
C^{p-var}\left( \left[ 0,T\right] ,G\right) $.
\end{definition}

For $p\geq 1$, let $\left[ p\right] $ denote the largest integer which is
less or equal to $p$.

\begin{definition}[branched $p$-rough path]
For $p\geq 1$, $X:\left[ 0,T\right] \rightarrow G_{\mathcal{L}}^{\left[ p%
\right] }$ is a branched $p$-rough path if $X$ is continuous and of finite $%
p $-variation.
\end{definition}

Let $L\left( 
%TCIMACRO{\U{211d} }%
%BeginExpansion
\mathbb{R}
%EndExpansion
^{d},%
%TCIMACRO{\U{211d} }%
%BeginExpansion
\mathbb{R}
%EndExpansion
^{e}\right) $ denote the set of continuous linear mappings from $%
%TCIMACRO{\U{211d} }%
%BeginExpansion
\mathbb{R}
%EndExpansion
^{d}$ to $%
%TCIMACRO{\U{211d} }%
%BeginExpansion
\mathbb{R}
%EndExpansion
^{e}$. We assume $Lip\left( \gamma \right) $ vector fields and their norms
are defined as in \cite[p.230, Definition 1.2.4]{lyons1998differential}. The
following theorem is the main result of the current paper.

\begin{theorem}
\label{Theorem main theorem}For $\gamma >p\geq 1$ and $i=1,2$, suppose $%
f^{i}:%
%TCIMACRO{\U{211d} }%
%BeginExpansion
\mathbb{R}
%EndExpansion
^{e}\rightarrow L\left( 
%TCIMACRO{\U{211d} }%
%BeginExpansion
\mathbb{R}
%EndExpansion
^{d},%
%TCIMACRO{\U{211d} }%
%BeginExpansion
\mathbb{R}
%EndExpansion
^{e}\right) $ are $Lip\left( \gamma \right) $ vector fields and $X^{i}:\left[
0,T\right] \rightarrow G_{\mathcal{L}}^{\left[ p\right] }$ are branched $p$%
-rough paths over $%
%TCIMACRO{\U{211d} }%
%BeginExpansion
\mathbb{R}
%EndExpansion
^{d}$. For $\xi ^{i}\in 
%TCIMACRO{\U{211d} }%
%BeginExpansion
\mathbb{R}
%EndExpansion
^{e}$, $i=1,2$, let $y^{i}$ denote the unique solution of the branched rough
differential equation:%
\begin{equation*}
dy_{t}^{i}=f^{i}\left( y_{t}^{i}\right) dX_{t}^{i},\text{ }y_{0}^{i}=\xi
^{i}.
\end{equation*}%
Denote $\lambda :=\max_{i=1,2}\left\vert f^{i}\right\vert _{Lip\left( \gamma
\right) }$, $\omega \left( s,t\right) :=\sum_{i=1,2}\left\Vert
X^{i}\right\Vert _{p-var,\left[ s,t\right] }^{p}$ and%
\begin{equation*}
\rho _{p-\omega ;\left[ 0,T\right] }\left( X^{1},X^{2}\right) :=\max_{\tau
\in \mathcal{F}_{\mathcal{L}}^{\left[ p\right] }}\sup_{0\leq s<t\leq T}\frac{%
\left\vert \left( X_{s,t}^{1},\tau \right) -\left( X_{s,t}^{2},\tau \right)
\right\vert }{\omega \left( s,t\right) ^{\frac{\left\vert \tau \right\vert }{%
p}}}.
\end{equation*}%
Then there exists a constant $M>0$ that only depends on $\gamma ,p,d$ such
that%
\begin{eqnarray*}
&&\sup_{0\leq s<t\leq T}\frac{\left\vert \left( y_{t}^{1}-y_{s}^{1}\right)
-\left( y_{t}^{2}-y_{s}^{2}\right) \right\vert }{\omega \left( s,t\right) ^{%
\frac{1}{p}}} \\
&\leq &M\lambda \left( \left\vert \xi ^{1}-\xi ^{2}\right\vert +\lambda
^{-1}\left\vert f^{1}-f^{2}\right\vert _{Lip\left( \gamma -1\right) }+\rho
_{p-\omega ;\left[ 0,T\right] }\left( X^{1},X^{2}\right) \right) \exp \left(
M\lambda ^{p}\omega \left( 0,T\right) \right) .
\end{eqnarray*}
\end{theorem}

The existence and uniqueness of the solution when the vector field is $%
Lip\left( \gamma \right) $ for $\gamma >p$ follow from \cite[Theorem 22]%
{lyons2015theory}. The $\rho _{p-\omega ;\left[ 0,T\right] }$ distance is
consistent with the $d_{\gamma }$-H\"{o}lder distance defined by Gubinelli 
\cite[p.710]{gubinelli2010ramification} where $\omega \left( s,t\right)
=\left\vert t-s\right\vert $.

Based on an isomorphism between branched rough paths and a class of $\Pi $%
-rough paths \cite{boedihardjo2019isomorphism}, our proof relies on an
inhomogeneous geodesic technique which extends the sub-Riemannian geometry
for geometric rough paths \cite{friz2008euler, friz2010multidimensional} to
branched rough paths.

Comparing with the current existing results \cite[Theorem 8.8]%
{gubinelli2010ramification} and \cite[Theorem 4.20]{friz2018differential},
our estimate only requires that the vector field is $Lip\left( \gamma
\right) $ for $\gamma >p$ while not $Lip\left( \left[ p\right] +1\right) $.
Moreover, our result specifies explicitly the Lipschitz dependence of the
solution with respect to the initial value, the vector field and the driving
branched rough path with the constant only depending on $\gamma ,p,d$.

\section{Proof}

In \cite{grossman1989hopf}, Grossman and Larson described several Hopf
algebras associated with certain family of trees. By deleting the additional
root, we call the Hopf algebra of non-planar forests with product \cite[$%
\left( 3.1\right) $]{grossman1989hopf} and coproduct \cite[p.199]%
{grossman1989hopf} the Grossman Larson Hopf algebra. Based on Foissy \cite[%
Section 8]{foissy2002finite} and Chapoton \cite{chapoton2010free}, Grossman
Larson algebra is freely generated by a collection of unlabeled trees.
Denote this collection of trees as $\mathcal{B}$. Denote the $\mathcal{L}$%
-labeled version of $\mathcal{B}$ as $\mathcal{B}_{\mathcal{L}}$ with $%
\mathcal{L}=\left\{ 1,2,\dots ,d\right\} $.

\begin{notation}
\label{Notation BL[p]}Let $\mathcal{B}_{\mathcal{L}}^{\left[ p\right]
}=\left\{ \nu _{1},\nu _{2},\dots ,\nu _{K}\right\} $ denote the set of
elements in $\mathcal{B}_{\mathcal{L}}$ of degree $1,\dots ,\left[ p\right] $%
.
\end{notation}

Then $K$ only depends on $p,d$.

\begin{notation}
Let $\mathcal{W}$ denote the set of finite sequences $k_{1}\cdots k_{m}$ for 
$k_{j}\in \left\{ 1,2,\dots ,K\right\} $, $j=1,2,\dots ,m$, including the
empty sequence denoted as $\epsilon $. For $k_{1}\cdots k_{m}\in \mathcal{W}$%
, define its degree%
\begin{equation*}
\left\Vert k_{1}\cdots k_{m}\right\Vert :=\left\vert \nu _{k_{1}}\right\vert
+\cdots +\left\vert \nu _{k_{m}}\right\vert
\end{equation*}%
where $\left\vert \nu _{j}\right\vert $ denotes the number of vertices in $%
\nu _{j}$ and $\left\Vert \epsilon \right\Vert :=0$.
\end{notation}

The set of infinite tensor series generated by $\mathcal{B}_{\mathcal{L}}^{%
\left[ p\right] }$ with the operation of tensor product forms an algebra. An
element $a$ of the algebra can be represented as $a=\sum_{w\in \mathcal{W}%
}\left( a,w\right) w$ for $\left( a,w\right) \in 
%TCIMACRO{\U{211d} }%
%BeginExpansion
\mathbb{R}
%EndExpansion
$. For $n=0,1,2,\dots $, the set $\sum_{w\in \mathcal{W},\left\Vert
w\right\Vert >n}c_{w}w$ for $c_{w}\in 
%TCIMACRO{\U{211d} }%
%BeginExpansion
\mathbb{R}
%EndExpansion
$ forms an ideal. Denote the quotient algebra as $\mathcal{A}^{n}$. Let $%
\mathfrak{G}$ denote the group of algebraic exponentials of Lie series
generated by $\left\{ 1,2,\dots ,K\right\} $ ($\mathfrak{G}$ is a group
based on Baker--Campbell--Hausdorff formula). Denote the group 
\begin{equation*}
\mathfrak{G}^{n}:=\mathfrak{G}\cap \mathcal{A}^{n}
\end{equation*}%
and denote the projection%
\begin{equation*}
\pi _{n}:\mathfrak{G}\rightarrow \mathfrak{G}^{n}.
\end{equation*}%
We equip $a\in \mathfrak{G}^{n}$ with the norm%
\begin{equation*}
\left\Vert a\right\Vert :=\sum_{w\in \mathcal{W},0<\left\Vert w\right\Vert
\leq n}\left\vert \left( a,w\right) \right\vert ^{\frac{1}{\left\Vert
w\right\Vert }}.
\end{equation*}%
$\mathfrak{G}^{n}$ is an inhomogeneous counterpart of the step-$n$ free
nilpotent Lie group \cite[p.235, Theorem 2.1.1]{lyons1998differential}.

\begin{notation}
Suppose $x=\left( x^{1},\dots ,x^{K}\right) :\left[ 0,T\right] \rightarrow 
%TCIMACRO{\U{211d} }%
%BeginExpansion
\mathbb{R}
%EndExpansion
^{K}$ is a continuous bounded variation path. For $n=0,1,\dots $ and $0\leq
s\leq t\leq T$, define $S_{n}\left( x\right) _{s,t}\in \mathfrak{G}^{n}$ as,
for $k_{1}\cdots k_{m}\in \mathcal{W}$, $\left\Vert k_{1}\cdots
k_{m}\right\Vert \leq n$, 
\begin{equation*}
\left( S_{n}\left( x\right) _{s,t},k_{1}\cdots k_{m}\right)
:=\idotsint\limits_{s<u_{1}<\cdots <u_{m}<t}dx_{u_{1}}^{k_{1}}\cdots
dx_{u_{m}}^{k_{m}}
\end{equation*}%
with $\left( S_{n}\left( x\right) _{s,t},\epsilon \right) :=1$.
\end{notation}

$S_{n}\left( x\right) $ is an inhomogeneous counterpart of the step-$n$
signature \cite[Definition 1.1]{hambly2010uniqueness}. The following Lemma
is an inhomogeneous generalization of Proposition 7.64 \cite%
{friz2010multidimensional}.

\begin{lemma}
\label{Lemma bound on paths related Pi rough paths}For $i=1,2$, $C>0$, $%
\delta >0$ and an integer $n\geq 1$, suppose $h^{i}\in \mathfrak{G}^{n}$, $%
\left\Vert h^{i}\right\Vert \leq C$ and%
\begin{equation*}
\max_{w\in \mathcal{W},\left\Vert w\right\Vert \leq n}\left\vert \left(
h^{1}-h^{2},w\right) \right\vert \leq \delta .
\end{equation*}%
Then there exist $x^{i}\in C^{1-var}\left( \left[ 0,1\right] ,%
%TCIMACRO{\U{211d} }%
%BeginExpansion
\mathbb{R}
%EndExpansion
^{K}\right) ,i=1,2$ such that%
\begin{equation*}
S_{n}\left( x^{i}\right) _{0,1}=h^{i},\text{ }i=1,2
\end{equation*}%
and a constant $M=M\left( C,p,d,n\right) >0$ such that%
\begin{equation*}
\max_{i=1,2}\left\Vert x^{i}\right\Vert _{1-var,\left[ 0,1\right] }\leq M%
\text{ and }\left\Vert x^{1}-x^{2}\right\Vert _{1-var,\left[ 0,1\right]
}\leq \delta M.
\end{equation*}
\end{lemma}

\begin{proof}
In the following proof, the constant $M$ may depend on $C,p,d,n$ and its
exact value may change.

Firstly, assume $\left( h^{1},w\right) =\left( h^{2},w\right) =0$ for $w\in 
\mathcal{W}$, $\left\Vert w\right\Vert =1,\dots ,n-1$. Then $h^{i}=1+l^{i}$
for $i=1,2$ with $l^{i}$ a homogeneous element of degree $n$ and $%
l_{2}=l_{1}+\delta m$ with $\left\Vert m\right\Vert \leq M$. Based on
similar proof as that of \cite[Theorem 7.32]{friz2010multidimensional} and 
\cite[Theorem 7.44]{friz2010multidimensional}, there exists $z\in
C^{1-var}\left( \left[ 0,1\right] ,%
%TCIMACRO{\U{211d} }%
%BeginExpansion
\mathbb{R}
%EndExpansion
^{K}\right) $ such that $S_{n}\left( z\right) _{0,1}=1+l^{1}-m$ and $%
\left\Vert z\right\Vert _{1-var,\left[ 0,1\right] }\leq M$. Similarly, there
exists $y=\left( y^{i}\right) _{i=1}^{K}\in C^{1-var}\left( \left[ 0,1\right]
,%
%TCIMACRO{\U{211d} }%
%BeginExpansion
\mathbb{R}
%EndExpansion
^{K}\right) $ such that $S_{n}\left( y\right) _{0,1}=1+m$ and $\left\Vert
y\right\Vert _{1-var,\left[ 0,1\right] }\leq M$. Let $x^{1}\ $be the
concatenation of $z$ and $y$ and let $x^{2}\ $be the concatenation of $z$
with $\widetilde{y}:=\left( \left( 1+\delta \right) ^{\left\vert \nu
_{i}\right\vert /n}y^{i}\right) _{i=1}^{K}$. Since $n\geq \left\vert \nu
_{j}\right\vert $ ($\mathfrak{G}^{n}$ does not involve $j$ when $\left\vert
\nu _{j}\right\vert >n$), we have $\left( 1+\delta \right) ^{\left\vert \nu
_{j}\right\vert /n}-1\leq \delta $ and%
\begin{equation*}
\left\Vert x^{1}-x^{2}\right\Vert _{1-var,\left[ 0,2\right] }\leq \delta
\left\Vert y\right\Vert _{1-var,\left[ 1,2\right] }.
\end{equation*}%
The first case is proved.

General case: we provide an inductive proof. The case $n=1$ follows from the
first case. Assuming the statement holds for elements in $\mathfrak{G}^{n}$,
we now prove that it holds for elements in $\mathfrak{G}^{n+1}$. By the
inductive hypothesis, there exist continuous bounded variation paths $z^{i}:%
\left[ 0,1\right] \rightarrow 
%TCIMACRO{\U{211d} }%
%BeginExpansion
\mathbb{R}
%EndExpansion
^{K}$, $i=1,2$ such that $S_{n}\left( z^{i}\right) _{0,1}=\pi _{n}\left(
h^{i}\right) $, $i=1,2$,%
\begin{equation*}
\max_{i=1,2}\left\Vert z^{i}\right\Vert _{1-var,\left[ 0,1\right] }\leq M%
\text{ and }\left\Vert z^{1}-z^{2}\right\Vert _{1-var,\left[ 0,1\right]
}\leq \delta M.
\end{equation*}%
Denote 
\begin{equation*}
k^{i}:=b^{i}\otimes h^{i}\text{ with }b^{i}:=S_{n+1}\left( \overleftarrow{%
z^{i}}\right) \text{, }i=1,2\text{,}
\end{equation*}%
where $\overleftarrow{z^{i}}$ denotes the time reversal of $z^{i}$. Then for 
$i=1,2$, $\left\Vert k^{i}\right\Vert \leq M$ and $\left( k^{i},w\right) =0$
for $w\in \mathcal{W}$, $\left\Vert w\right\Vert =1,\dots ,n$. For $w\in 
\mathcal{W}$, $\left\Vert w\right\Vert =n+1$,%
\begin{equation*}
\left\vert \left( k^{1}-k^{2},w\right) \right\vert \leq \sum_{uv=w}\left(
\left\vert \left( b^{1},u\right) \right\vert \left\vert \left(
h^{1},v\right) -\left( h^{2},v\right) \right\vert +\left\vert \left(
b^{1},u\right) -\left( b^{2},u\right) \right\vert \left\vert \left(
h^{2},v\right) \right\vert \right) ,
\end{equation*}%
where $uv$ denotes the concatenation of $u$ and $v$. Since iterated
integrals are continuous in $1$-variation of the underlying path, combined
with the conditions on $h^{i}$, we have $\left\vert \left(
k^{1}-k^{2},w\right) \right\vert \leq \delta M$ for $w\in \mathcal{W}$, $%
\left\Vert w\right\Vert =n+1$. Based on the first case, there exist
continuous bounded variation paths $y^{i}:\left[ 0,1\right] \rightarrow 
%TCIMACRO{\U{211d} }%
%BeginExpansion
\mathbb{R}
%EndExpansion
^{K},i=1,2$ such that%
\begin{equation*}
S_{n+1}\left( y^{i}\right) _{0,1}=k^{i},\text{ }i=1,2
\end{equation*}%
and%
\begin{equation*}
\max_{i=1,2}\left\Vert y^{1}\right\Vert _{1-var,\left[ 0,1\right] }\leq M%
\text{ and }\left\Vert y^{1}-y^{2}\right\Vert _{1-var,\left[ 0,1\right]
}\leq \delta M.
\end{equation*}%
For $i=1,2$, let $x^{i}$ be the concatenation of $z^{i}$ with $y^{i}$. The
proof is finished.
\end{proof}

Based on \cite{foissy2002finite, chapoton2010free}, Grossman Larson Hopf
algebra is isomorphic as a Hopf algebra to the Tensor Hopf algebra generated
by a collection of trees. By deleting the additional root, we assume
Grossman Larson Hopf algebra with product \cite[$\left( 3.1\right) $]%
{grossman1989hopf} and coproduct \cite[p.199]{grossman1989hopf} is a Hopf
algebra of forests. Denote the degree-$n$ truncated group of group-like
elements in Grossman Larson Hopf algebra as $\mathcal{G}_{\mathcal{L}}^{n}$.

\begin{notation}
\label{Notation isomorphism between Pi rough paths and branched rough paths}%
Denote the group isomorphism $\Phi :\mathfrak{G}^{\left[ p\right]
}\rightarrow \mathcal{G}_{\mathcal{L}}^{\left[ p\right] }$.
\end{notation}

\begin{lemma}
\label{Lemma existence of continuous bounded variation paths for Grossman
Larson group}For $i=1,2$, $C>0$ and $\delta >0$, suppose $g^{i}\in \mathcal{G%
}_{\mathcal{L}}^{\left[ p\right] }$, $\left\Vert g^{i}\right\Vert \leq C$ and%
\begin{equation*}
\max_{\tau \in \mathcal{F}_{\mathcal{L}}^{\left[ p\right] }}\left\vert
\left( g^{1}-g^{2},\tau \right) \right\vert \leq \delta .
\end{equation*}%
Then there exist $x^{i}\in C^{1-var}\left( \left[ 0,1\right] ,%
%TCIMACRO{\U{211d} }%
%BeginExpansion
\mathbb{R}
%EndExpansion
^{K}\right) $, $i=1,2$ such that%
\begin{equation*}
\Phi \left( S_{\left[ p\right] }\left( x^{i}\right) _{0,1}\right) =g^{i},%
\text{ }i=1,2
\end{equation*}%
and a constant $M=M\left( C,p,d\right) >0$ such that%
\begin{equation*}
\max_{i=1,2}\left\Vert x^{i}\right\Vert _{1-var,\left[ 0,1\right] }\leq M%
\text{ and }\left\Vert x^{1}-x^{2}\right\Vert _{1-var,\left[ 0,1\right]
}\leq \delta M.
\end{equation*}
\end{lemma}

\begin{proof}
Denote $h^{i}:=\Phi ^{-1}\left( g^{i}\right) $, $i=1,2$. By $\left\Vert
g^{i}\right\Vert \leq C$, we have $\left\Vert h^{i}\right\Vert \leq M$, $%
i=1,2$ and%
\begin{equation*}
\sup_{w\in \mathcal{W},\left\Vert w\right\Vert \leq \left[ p\right]
}\left\vert \left( h^{1}-h^{2},w\right) \right\vert \leq M\max_{\tau \in 
\mathcal{F}_{\mathcal{L}}^{\left[ p\right] }}\left\vert \left(
g^{1}-g^{2},\tau \right) \right\vert \leq M\delta \text{.}
\end{equation*}%
Then the statement holds based on Lemma \ref{Lemma bound on paths related Pi
rough paths}.
\end{proof}

For $a\in \mathcal{L}$, denote by $\bullet _{a}$ the tree that has one
vertex and a label $a\in \mathcal{L}$ on the vertex. For $\mathcal{L}$%
-labeled trees $\left\{ \tau _{i}\right\} _{i=1}^{k}$ and a label $a\in 
\mathcal{L}$, denote by $\left[ \tau _{1}\cdots \tau _{k}\right] _{a}$ the
labeled tree obtained by grafting the roots of $\left\{ \tau _{i}\right\}
_{i=1}^{k}$ to a new root with a label $a\in \mathcal{L}$ on the new root.
Then $\left\vert \left[ \tau _{1}\cdots \tau _{k}\right] _{a}\right\vert
=\sum_{i=1}^{k}\left\vert \tau _{i}\right\vert +1$.

\begin{notation}
\label{Notation of f(tau)}For\ sufficiently smooth\ $f=\left( f_{1},\dots
,f_{d}\right) :%
%TCIMACRO{\U{211d} }%
%BeginExpansion
\mathbb{R}
%EndExpansion
^{e}\rightarrow L\left( 
%TCIMACRO{\U{211d} }%
%BeginExpansion
\mathbb{R}
%EndExpansion
^{d},%
%TCIMACRO{\U{211d} }%
%BeginExpansion
\mathbb{R}
%EndExpansion
^{e}\right) $, define $f:\mathcal{T}_{\mathcal{L}}\rightarrow \left( 
%TCIMACRO{\U{211d} }%
%BeginExpansion
\mathbb{R}
%EndExpansion
^{e}\rightarrow 
%TCIMACRO{\U{211d} }%
%BeginExpansion
\mathbb{R}
%EndExpansion
^{e}\right) $ inductively as, for $a\in \mathcal{L}$ and $\tau _{i}\in 
\mathcal{T}_{\mathcal{L}}$, $i=1,\dots ,k$,%
\begin{equation*}
f\left( \bullet _{a}\right) :=f_{a}\text{ and }f\left( \left[ \tau
_{1}\cdots \tau _{k}\right] _{a}\right) :=\left( d^{k}f_{a}\right) \left(
f\left( \tau _{1}\right) \cdots f\left( \tau _{k}\right) \right)
\end{equation*}%
where $d^{k}f_{a}$ denotes the $k$-th Fr\'{e}chet derivative of $f_{a}$.
\end{notation}

Suppose $x\in C^{1-var}\left( \left[ 0,T\right] ,%
%TCIMACRO{\U{211d} }%
%BeginExpansion
\mathbb{R}
%EndExpansion
^{K}\right) $, $f:%
%TCIMACRO{\U{211d} }%
%BeginExpansion
\mathbb{R}
%EndExpansion
^{e}\rightarrow L\left( 
%TCIMACRO{\U{211d} }%
%BeginExpansion
\mathbb{R}
%EndExpansion
^{K},%
%TCIMACRO{\U{211d} }%
%BeginExpansion
\mathbb{R}
%EndExpansion
^{e}\right) $ is $Lip\left( 1\right) $ and $\xi \in 
%TCIMACRO{\U{211d} }%
%BeginExpansion
\mathbb{R}
%EndExpansion
^{e}$. Denote by 
\begin{equation*}
\pi _{f}\left( 0,\xi ;x\right)
\end{equation*}%
the unique solution to the ODE%
\begin{equation*}
dy_{t}=f\left( y_{t}\right) dx_{t},\text{ }y_{0}=\xi \text{.}
\end{equation*}%
For $f_{j}:%
%TCIMACRO{\U{211d} }%
%BeginExpansion
\mathbb{R}
%EndExpansion
^{e}\rightarrow 
%TCIMACRO{\U{211d} }%
%BeginExpansion
\mathbb{R}
%EndExpansion
^{e}$, denote 
\begin{equation*}
\left\vert f_{j}\right\vert _{\infty }:=\sup_{y\in 
%TCIMACRO{\U{211d} }%
%BeginExpansion
\mathbb{R}
%EndExpansion
^{e}}\left\vert f_{j}\left( y\right) \right\vert .
\end{equation*}
For $y:\left[ 0,T\right] \rightarrow 
%TCIMACRO{\U{211d} }%
%BeginExpansion
\mathbb{R}
%EndExpansion
^{e}$ and $0\leq s\leq t\leq T$, denote 
\begin{equation*}
y_{s,t}:=y_{t}-y_{s}.
\end{equation*}

\begin{proposition}
\label{Proposition ODE estimates}Assume that

(i) $f=\left( f_{1},\dots ,f_{K}\right) :%
%TCIMACRO{\U{211d} }%
%BeginExpansion
\mathbb{R}
%EndExpansion
^{e}\rightarrow L\left( 
%TCIMACRO{\U{211d} }%
%BeginExpansion
\mathbb{R}
%EndExpansion
^{K},%
%TCIMACRO{\U{211d} }%
%BeginExpansion
\mathbb{R}
%EndExpansion
^{e}\right) $ and $\widetilde{f}=\left( \widetilde{f}_{1},\dots ,\widetilde{f%
}_{K}\right) :%
%TCIMACRO{\U{211d} }%
%BeginExpansion
\mathbb{R}
%EndExpansion
^{e}\rightarrow L\left( 
%TCIMACRO{\U{211d} }%
%BeginExpansion
\mathbb{R}
%EndExpansion
^{K},%
%TCIMACRO{\U{211d} }%
%BeginExpansion
\mathbb{R}
%EndExpansion
^{e}\right) $ are $Lip\left( 1\right) $. For $j=1,\dots ,K$, denote 
\begin{equation*}
M_{j}:=\max \left\{ \left\vert f_{j}\right\vert _{Lip\left( 1\right)
},\left\vert \widetilde{f}_{j}\right\vert _{Lip\left( 1\right) }\right\} .
\end{equation*}

(ii) $x=\left( x^{1},\dots ,x^{K}\right) $ and $\widetilde{x}=\left( 
\widetilde{x}^{1},\dots ,\widetilde{x}^{K}\right) $ are in $C^{1-var}\left( %
\left[ 0,T\right] ,%
%TCIMACRO{\U{211d} }%
%BeginExpansion
\mathbb{R}
%EndExpansion
^{K}\right) $. For $j=1,\dots ,K$, denote 
\begin{equation*}
l_{j}:=\max \left\{ \left\Vert x^{j}\right\Vert _{1-var,\left[ 0,T\right]
},\left\Vert \widetilde{x}^{j}\right\Vert _{1-var,\left[ 0,T\right]
}\right\} .
\end{equation*}

(iii) $y_{0},\widetilde{y}_{0}\in 
%TCIMACRO{\U{211d} }%
%BeginExpansion
\mathbb{R}
%EndExpansion
^{e}$ are initial values.

Denote $y=\pi _{f}\left( 0,y_{0};x\right) $ and $\widetilde{y}=\pi _{%
\widetilde{f}}\left( 0,\widetilde{y}_{0};\widetilde{x}\right) $. Then%
\begin{eqnarray}
&&\sup_{t\in \left[ 0,T\right] }\left\vert y_{0,t}-\widetilde{y}%
_{0,t}\right\vert  \label{ODE estimate 1} \\
&\leq &\sum_{j=1}^{K}\left( M_{j}l_{j}\left\vert y_{0}-\widetilde{y}%
_{0}\right\vert +M_{j}\left\Vert x^{j}-\widetilde{x}^{j}\right\Vert _{1-var,%
\left[ 0,T\right] }+l_{j}\left\vert f_{j}-\widetilde{f}_{j}\right\vert
_{\infty }\right) \exp \left( 2\sum_{j=1}^{K}M_{j}l_{j}\right)  \notag
\end{eqnarray}%
and%
\begin{eqnarray}
&&\sup_{t\in \left[ 0,T\right] }\left\vert y_{t}-\widetilde{y}_{t}\right\vert
\label{ODE estimate 2} \\
&\leq &\left( \left\vert y_{0}-\widetilde{y}_{0}\right\vert
+\sum_{j=1}^{K}M_{j}\left\Vert x^{j}-\widetilde{x}^{j}\right\Vert _{1-var,%
\left[ 0,T\right] }+\sum_{j=1}^{K}l_{j}\left\vert f_{j}-\widetilde{f}%
_{j}\right\vert _{\infty }\right) \exp \left(
2\sum_{j=1}^{K}M_{j}l_{j}\right) .  \notag
\end{eqnarray}
\end{proposition}

\begin{proof}
Without loss of generality, assume $x_{0}=\widetilde{x}_{0}=0$. Since%
\begin{equation*}
\int_{0}^{t}f_{j}\left( \widetilde{y}_{r}\right) d\left( x_{r}^{j}-%
\widetilde{x}_{r}^{j}\right) =f_{j}\left( \widetilde{y}_{t}\right) \left(
x_{t}^{j}-\widetilde{x}_{t}^{j}\right) -\int_{0}^{t}\left( x_{r}^{j}-%
\widetilde{x}_{r}^{j}\right) df_{j}\left( \widetilde{y}_{r}\right) ,
\end{equation*}%
we have%
\begin{eqnarray*}
&&\left\vert y_{0,t}-\widetilde{y}_{0,t}\right\vert \\
&\leq &\left\vert y_{0}-\widetilde{y}_{0}\right\vert
\sum_{j=1}^{K}M_{j}l_{j}+\sum_{j=1}^{K}M_{j}\int_{0}^{t}\left\vert y_{0,r}-%
\widetilde{y}_{0,r}\right\vert \left\vert dx_{r}^{j}\right\vert
+\sum_{j=1}^{K}l_{j}\left\vert f_{j}-\widetilde{f}_{j}\right\vert _{\infty }
\\
&&+\left( 1+\sum_{j=1}^{K}M_{j}l_{j}\right) \sum_{j=1}^{K}M_{j}\sup_{t\in 
\left[ 0,T\right] }\left\vert x_{t}^{j}-\widetilde{x}_{t}^{j}\right\vert .
\end{eqnarray*}%
Since $x_{0}=\widetilde{x}_{0}=0$, we have $\sup_{t\in \left[ 0,T\right]
}\left\vert x_{t}^{j}-\widetilde{x}_{t}^{j}\right\vert \leq \left\Vert x^{j}-%
\widetilde{x}^{j}\right\Vert _{1-var,\left[ 0,T\right] }$. Based on
Gronwall's Lemma, the first inequality holds. The second inequality can be
proved similarly.
\end{proof}

For $\gamma >0$, let $\lfloor \gamma \rfloor $ denote the largest integer
which is strictly less than $\gamma $. Denote $I\left( x\right) :=x$ for $%
x\in 
%TCIMACRO{\U{211d} }%
%BeginExpansion
\mathbb{R}
%EndExpansion
^{e}$. Recall that $\epsilon $ denotes the empty element in $\mathcal{W}$.
For $f^{i}:%
%TCIMACRO{\U{211d} }%
%BeginExpansion
\mathbb{R}
%EndExpansion
^{e}\rightarrow L\left( 
%TCIMACRO{\U{211d} }%
%BeginExpansion
\mathbb{R}
%EndExpansion
^{d},%
%TCIMACRO{\U{211d} }%
%BeginExpansion
\mathbb{R}
%EndExpansion
^{e}\right) $, $i=1,2$ and $\nu \in \mathcal{B}_{\mathcal{L}}^{\left[ p%
\right] }$ in Notation \ref{Notation BL[p]}, denote $f^{i}\left( \nu \right) 
$ as in Notation \ref{Notation of f(tau)}.

\begin{notation}
\label{Notation F_i^w}Suppose $f^{i}:%
%TCIMACRO{\U{211d} }%
%BeginExpansion
\mathbb{R}
%EndExpansion
^{e}\rightarrow L\left( 
%TCIMACRO{\U{211d} }%
%BeginExpansion
\mathbb{R}
%EndExpansion
^{d},%
%TCIMACRO{\U{211d} }%
%BeginExpansion
\mathbb{R}
%EndExpansion
^{e}\right) $, $i=1,2$ are $Lip\left( \gamma \right) $ for some $\gamma >1$.
For $k_{1}\cdots k_{m}\in \mathcal{W}$, $\left\Vert k_{1}\cdots
k_{m}\right\Vert \leq \lfloor \gamma \rfloor $, define inductively%
\begin{equation*}
F_{i}^{\epsilon }:=I\text{ and }F_{i}^{k_{1}\cdots
k_{m}}:=dF_{i}^{k_{2}\cdots k_{m}}\left( f^{i}\left( \nu _{k_{1}}\right)
\right) \text{, }i=1,2\text{,}
\end{equation*}%
where $dF_{i}^{k_{2}\cdots k_{m}}$ denotes the Fr\'{e}chet derivative of $%
F_{i}^{k_{2}\cdots k_{m}}$.
\end{notation}

The following simple Lemma is helpful when estimating the increments of
functions.

\begin{lemma}
\label{Lemma increments of functions}For $i=1,2$, suppose $q^{i}:%
%TCIMACRO{\U{211d} }%
%BeginExpansion
\mathbb{R}
%EndExpansion
^{e}\rightarrow 
%TCIMACRO{\U{211d} }%
%BeginExpansion
\mathbb{R}
%EndExpansion
$ and $r^{i}:%
%TCIMACRO{\U{211d} }%
%BeginExpansion
\mathbb{R}
%EndExpansion
^{e}\rightarrow 
%TCIMACRO{\U{211d} }%
%BeginExpansion
\mathbb{R}
%EndExpansion
$. For $a,b\in 
%TCIMACRO{\U{211d} }%
%BeginExpansion
\mathbb{R}
%EndExpansion
^{e}$, 
\begin{eqnarray*}
&&\left( q^{1}r^{1}-q^{2}r^{2}\right) \left( a\right) -\left(
q^{1}r^{1}-q^{2}r^{2}\right) \left( b\right) \\
&=&\left( q^{1}\left( r^{1}-r^{2}\right) \right) \left( a\right) -\left(
q^{1}\left( r^{1}-r^{2}\right) \right) \left( b\right) \\
&&+\left( \left( q^{1}-q^{2}\right) r^{2}\right) \left( a\right) -\left(
\left( q^{1}-q^{2}\right) r^{2}\right) \left( b\right) \\
&=&:Q\left( a\right) -Q\left( b\right) +R\left( a\right) -R\left( b\right)
\end{eqnarray*}%
where $Q:=q^{1}\left( r^{1}-r^{2}\right) $ and $R:=\left( q^{1}-q^{2}\right)
r^{2}$.
\end{lemma}

Lemma \ref{Lemma Abar} and Lemma \ref{Lemma Bbar} below are generalizations
of Lemma 10.23 \cite{friz2010multidimensional} and Lemma 10.25 \cite%
{friz2010multidimensional} respectively and apply to ODEs with inhomogeneous
drivers. Recall $\mathcal{B}_{\mathcal{L}}^{\left[ p\right] }=\left\{ \nu
_{1},\nu _{2},\dots ,\nu _{K}\right\} $ in Notation \ref{Notation BL[p]}.
Since $K$ denotes the number of elements in $\mathcal{B}_{\mathcal{L}}^{%
\left[ p\right] }$, $K$ only depends on $p,d$.

\begin{lemma}
\label{Lemma Abar}Fix $\gamma >p\geq 1$.

(i) Suppose $f^{i}:%
%TCIMACRO{\U{211d} }%
%BeginExpansion
\mathbb{R}
%EndExpansion
^{e}\rightarrow L\left( 
%TCIMACRO{\U{211d} }%
%BeginExpansion
\mathbb{R}
%EndExpansion
^{d},%
%TCIMACRO{\U{211d} }%
%BeginExpansion
\mathbb{R}
%EndExpansion
^{e}\right) $, $i=1,2$ are $Lip\left( \gamma \right) $. Denote $\lambda
:=\max_{i=1,2}\left\vert f^{i}\right\vert _{Lip\left( \gamma \right) }$.

(ii) For $i=1,2$, suppose $x^{i}=\left( x^{i,1},\dots ,x^{i,K}\right) $ and $%
\widetilde{x}^{i}=\left( \widetilde{x}^{i,1},\dots ,\widetilde{x}%
^{i,K}\right) $ are paths in $C^{1-var}\left( \left[ 0,1\right] ,%
%TCIMACRO{\U{211d} }%
%BeginExpansion
\mathbb{R}
%EndExpansion
^{K}\right) $ such that 
\begin{equation*}
S_{\left[ p\right] }\left( x^{i}\right) _{0,1}=S_{\left[ p\right] }\left( 
\widetilde{x}^{i}\right) _{0,1},\text{ }i=1,2.
\end{equation*}

(iii) For $C\geq 0$, $l\geq 0$ and $\delta \geq 0$, suppose for $j=1,\dots
,K $,%
\begin{eqnarray*}
\max_{i=1,2}\left\{ \left\Vert x^{i,j}\right\Vert _{1-var,\left[ 0,1\right]
},\left\Vert \widetilde{x}^{i,j}\right\Vert _{1-var,\left[ 0,1\right]
}\right\} &\leq &Cl^{\left\vert \nu _{j}\right\vert }, \\
\max \left\{ \left\Vert x^{1,j}-x^{2,j}\right\Vert _{1-var,\left[ 0,1\right]
},\left\Vert \widetilde{x}^{1,j}-\widetilde{x}^{2,j}\right\Vert _{1-var,%
\left[ 0,1\right] }\right\} &\leq &\delta Cl^{\left\vert \nu _{j}\right\vert
}.
\end{eqnarray*}

Denote vector fields $V^{i}:=\left( f^{i}\left( \nu _{1}\right) ,\dots
,f^{i}\left( \nu _{K}\right) \right) $, $i=1,2$. For $y_{0}^{i}\in 
%TCIMACRO{\U{211d} }%
%BeginExpansion
\mathbb{R}
%EndExpansion
^{e}$, $i=1,2$, denote $y^{i}:=\pi _{V^{i}}\left( 0,y_{0}^{i};x^{i}\right) $
and $\widetilde{y}^{i}:=\pi _{V^{i}}\left( 0,y_{0}^{i};\widetilde{x}%
^{i}\right) $, $i=1,2$. Then there exists a constant $M=M\left( C,\gamma
,p,d\right) >0$ such that, when $\lambda l\leq 1$,%
\begin{eqnarray*}
&&\left\vert \left( y_{0,1}^{1}-\widetilde{y}_{0,1}^{1}\right) -\left(
y_{0,1}^{2}-\widetilde{y}_{0,1}^{2}\right) \right\vert \\
&\leq &M\left( \lambda l\right) ^{\gamma }\left( \left\vert
y_{0}^{1}-y_{0}^{2}\right\vert +\delta +\lambda ^{-1}\left\vert
f^{1}-f^{2}\right\vert _{Lip\left( \gamma -1\right) }\right) .
\end{eqnarray*}
\end{lemma}

\begin{proof}
Without loss of generality, assume $\gamma \in (p,\left[ p\right] +1]$ and
denote $N:=\left[ p\right] $. The constant $M$ in the following proof may
depend on $C,\gamma ,p,d$ and its exact value may change.

First case, assume $\widetilde{x}^{1}=\widetilde{x}^{2}=0$ and we want to
estimate $\left\vert y_{0,1}^{1}-y_{0,1}^{2}\right\vert $. By iteratively
applying the fundamental theorem of calculus, for $i=1,2$,%
\begin{eqnarray*}
y_{0,1}^{i} &=&\sum_{\left\Vert k_{1}\cdots k_{m}\right\Vert =N}\quad
\idotsint\limits_{0<u_{1}<\cdots <u_{m}<1}\left( F_{i}^{k_{1}\cdots
k_{m}}\left( y_{u_{1}}^{i}\right) -F_{i}^{k_{1}\cdots k_{m}}\left(
y_{0}^{i}\right) \right) dx_{u_{1}}^{i,k_{1}}\cdots dx_{u_{m}}^{i,k_{m}} \\
&&+\sum_{\substack{ \left\Vert k_{1}\cdots k_{m}\right\Vert >N  \\ %
\left\Vert k_{2}\cdots k_{m}\right\Vert <N}}\quad
\idotsint\limits_{0<u_{1}<\cdots <u_{m}<1}F_{i}^{k_{1}\cdots k_{m}}\left(
y_{u_{1}}^{i}\right) dx_{u_{1}}^{i,k_{1}}\cdots dx_{u_{m}}^{i,k_{m}}
\end{eqnarray*}%
For $i=1,2$, denote $F_{i}^{N}:=\left( F_{i}^{w}\right) _{w\in \mathcal{W}%
,\left\Vert w\right\Vert =N}$ with $F_{i}^{w}$ in Notation \ref{Notation
F_i^w} and denote $x_{u,1}^{i,N}:=\left( x_{u,1}^{i,w}\right) _{w\in 
\mathcal{W},\left\Vert w\right\Vert =N}$ where 
\begin{equation*}
x_{u,1}^{i,k_{1}\cdots k_{m}}:=\idotsint\limits_{u<u_{1}<\cdots
<u_{m}<1}dx_{u_{1}}^{i,k_{1}}\cdots dx_{u_{m}}^{i,k_{m}}\text{ for }%
k_{1}\cdots k_{m}\in \mathcal{W}.
\end{equation*}%
Since $\lambda l\leq 1$, we have $\left\vert y_{0,\cdot }^{i}\right\vert
_{\infty ,\left[ 0,1\right] }\leq M\lambda l$, $i=1,2$. Separate the $%
Lip\left( \gamma -N+1\right) $ term $\left( d^{N-1}f^{i}\right) \left(
f^{i}\right) ^{N-1}$ from $F_{i}^{N}$, $i=1,2$ (the rest terms are $%
Lip\left( 2\right) $). Based on Lemma \ref{Lemma increments of functions},
by adapting the proof of Lemma 10.22 \cite{friz2010multidimensional} and
combining with assumption (iii), we have 
\begin{eqnarray}
&&\left\vert \int_{0}^{1}\left( F_{1}^{N}\left( y_{u}^{1}\right)
-F_{1}^{N}\left( y_{0}^{1}\right) \right) dx_{u,1}^{1,N}-\int_{0}^{1}\left(
F_{2}^{N}\left( y_{u}^{2}\right) -F_{2}^{N}\left( y_{0}^{2}\right) \right)
dx_{u,1}^{2,N}\right\vert  \label{inner estimate 1} \\
&\leq &M\left( \lambda l\right) ^{N}\left\vert y_{0,\cdot }^{1}-y_{0,\cdot
}^{2}\right\vert _{\infty ,\left[ 0,1\right] }  \notag \\
&&+M\left( \lambda l\right) ^{\gamma }\left( \left\vert
y_{0}^{1}-y_{0}^{2}\right\vert +\lambda ^{-1}\left\vert
f^{1}-f^{2}\right\vert _{Lip\left( \gamma -1\right) }\right)  \notag \\
&&+M\delta \left( \lambda l\right) ^{N+1}  \notag
\end{eqnarray}

For $j=1,\dots ,K$,%
\begin{equation*}
\left\vert f^{1}\left( \nu _{j}\right) -f^{2}\left( \nu _{j}\right)
\right\vert _{\infty }\leq M\lambda ^{\left\vert \nu _{j}\right\vert
-1}\left\vert f^{1}-f^{2}\right\vert _{Lip\left( \gamma -1\right) }\text{.}
\end{equation*}%
Since $\lambda l\leq 1$, based on $\left( \ref{ODE estimate 1}\right) $, we
have%
\begin{equation*}
\left\vert y_{0,\cdot }^{1}-y_{0,\cdot }^{2}\right\vert _{\infty ,\left[ 0,1%
\right] }\leq M\lambda l\left( \left\vert y_{0}^{1}-y_{0}^{2}\right\vert
+\delta +\lambda ^{-1}\left\vert f^{1}-f^{2}\right\vert _{Lip\left( \gamma
-1\right) }\right) .
\end{equation*}%
Putting the estimate into $\left( \ref{inner estimate 1}\right) $, we get%
\begin{eqnarray}
&&\left\vert \int_{0}^{1}\left( F_{1}^{N}\left( y_{u}^{1}\right)
-F_{1}^{N}\left( y_{0}^{1}\right) \right) dx_{u,1}^{1,N}-\int_{0}^{1}\left(
F_{2}^{N}\left( y_{u}^{2}\right) -F_{2}^{N}\left( y_{0}^{2}\right) \right)
dx_{u,1}^{2,N}\right\vert  \label{inner Lemma Abar 1} \\
&\leq &M\left( \lambda l\right) ^{\gamma }\left( \left\vert
y_{0}^{1}-y_{0}^{2}\right\vert +\delta +\lambda ^{-1}\left\vert
f^{1}-f^{2}\right\vert _{Lip\left( \gamma -1\right) }\right)  \notag
\end{eqnarray}

On the other hand, for $w\in \mathcal{W}$, 
\begin{eqnarray*}
&&\idotsint\limits_{0<u_{1}<\cdots <u_{m}<1}\left( F_{1}^{w}\left(
y_{u_{1}}^{1}\right) dx_{u,1}^{1,w}-F_{2}^{w}\left( y_{u_{1}}^{2}\right)
dx_{u,1}^{2,w}\right) \\
&=&\idotsint\limits_{0<u_{1}<\cdots <u_{m}<1}\left( F_{1}^{w}\left(
y_{u_{1}}^{1}\right) -F_{1}^{w}\left( y_{u_{1}}^{2}\right) \right)
dx_{u,1}^{1,w} \\
&&+\idotsint\limits_{0<u_{1}<\cdots <u_{m}<1}\left(
F_{1}^{w}-F_{2}^{w}\right) \left( y_{u_{1}}^{2}\right) dx_{u,1}^{1,w} \\
&&+\idotsint\limits_{0<u_{1}<\cdots <u_{m}<1}F_{2}^{w}\left(
y_{u_{1}}^{2}\right) \left( dx_{u,1}^{1,w}-dx_{u,1}^{2,w}\right) .
\end{eqnarray*}%
Suppose $w=kw_{1}$, where $k\in \left\{ 1,\dots ,K\right\} $ and $w,w_{1}\in 
\mathcal{W}$, $\left\Vert w\right\Vert >N$, $\left\Vert w_{1}\right\Vert <N$%
. Then $\max_{i=1,2}\left\vert F_{i}^{w}\right\vert _{Lip\left( 1\right)
}\leq M\lambda ^{\left\Vert w\right\Vert }$ and $\left\vert
F_{1}^{w}-F_{2}^{w}\right\vert _{\infty }\leq M\lambda ^{\left\Vert
w\right\Vert -1}\left\vert f^{1}-f^{2}\right\vert _{Lip\left( \gamma
-1\right) }$. Since $\lambda l\leq 1$, combined with $\left( \ref{ODE
estimate 2}\right) $ and assumption (iii), we have%
\begin{eqnarray}
&&\left\vert \quad \idotsint\limits_{0<u_{1}<\cdots <u_{m}<1}\left(
F_{1}^{w}\left( y_{u_{1}}^{1}\right) dx_{u,1}^{1,w}-F_{2}^{w}\left(
y_{u_{1}}^{2}\right) dx_{u,1}^{2,w}\right) \right\vert
\label{inner Lemma Abar 2} \\
&\leq &M\left( \lambda l\right) ^{N+1}\left( \left\vert
y_{0}^{1}-y_{0}^{2}\right\vert +\delta +\lambda ^{-1}\left\vert
f^{1}-f^{2}\right\vert _{Lip\left( \gamma -1\right) }\right) .  \notag
\end{eqnarray}

Since we assumed that $\lambda l\leq 1$, combine $\left( \ref{inner Lemma
Abar 1}\right) $ with $\left( \ref{inner Lemma Abar 2}\right) $, we have%
\begin{equation*}
\left\vert y_{0,1}^{1}-y_{0,1}^{2}\right\vert \leq M\left( \lambda l\right)
^{\gamma }\left( \left\vert y_{0}^{1}-y_{0}^{2}\right\vert +\delta +\lambda
^{-1}\left\vert f^{1}-f^{2}\right\vert _{Lip\left( \gamma -1\right) }\right)
.
\end{equation*}

General case. For $i=1,2$, let $z^{i}:=\overleftarrow{\widetilde{x}^{i}}%
\sqcup x^{i}$ be the concatenation of the time reversal of $\widetilde{x}%
^{i} $ with $x^{i}$. Reparametrize $z^{i}$ to be from $\left[ 0,1\right] $
to $%
%TCIMACRO{\U{211d} }%
%BeginExpansion
\mathbb{R}
%EndExpansion
^{K}$. Based on the assumption (ii) and (iii), $S_{\left[ p\right] }\left(
z^{i}\right) _{0,1}=1$, $i=1,2$ and $\max_{i=1,2}\left\Vert
z^{i,j}\right\Vert _{1-var,\left[ 0,1\right] }\leq 2Cl^{\left\vert \nu
_{j}\right\vert }$, $\left\Vert z^{1,j}-z^{2,j}\right\Vert _{1-var,\left[ 0,1%
\right] }\leq 2\delta Cl^{\left\vert \nu _{j}\right\vert }$, $j=1,\dots ,K$.
Since for $i=1,2$,%
\begin{equation*}
y_{0,1}^{i}-\widetilde{y}_{0,1}^{i}=\pi _{V^{i}}\left( 0,\pi _{V^{i}}\left(
0,y_{0}^{i};\widetilde{x}^{i}\right) _{1};z^{i}\right) _{0,1}\text{.}
\end{equation*}%
Then the result follows by applying the first case to $z^{i}$, $i=1,2$ and
combining with $\left( \ref{ODE estimate 2}\right) $.
\end{proof}

For $\gamma >1$, denote $\left\{ \gamma \right\} :=\gamma -\lfloor \gamma
\rfloor $.

\begin{lemma}
\label{Lemma Bbar}Fix $\gamma >p\geq 1$.

(i) For $i=1,2$, suppose $f^{i}:%
%TCIMACRO{\U{211d} }%
%BeginExpansion
\mathbb{R}
%EndExpansion
^{e}\rightarrow L\left( 
%TCIMACRO{\U{211d} }%
%BeginExpansion
\mathbb{R}
%EndExpansion
^{d},%
%TCIMACRO{\U{211d} }%
%BeginExpansion
\mathbb{R}
%EndExpansion
^{e}\right) $ are $Lip\left( \gamma \right) $. Denote $\lambda
:=\max_{i=1,2}\left\vert f^{i}\right\vert _{Lip\left( \gamma \right) }$.

(ii) For $i=1,2$, suppose $x^{i}=\left( x^{i,1},\dots ,x^{i,K}\right) \in
C^{1-var}\left( \left[ 0,1\right] ,%
%TCIMACRO{\U{211d} }%
%BeginExpansion
\mathbb{R}
%EndExpansion
^{K}\right) $ and there exist constants $C\geq 0,\delta \geq 0$ and $l\geq 0$
such that for $j=1,\dots ,K$,%
\begin{eqnarray*}
\max_{i=1,2}\left\Vert x^{i,j}\right\Vert _{1-var,\left[ 0,1\right] } &\leq
&Cl^{\left\vert \nu _{j}\right\vert }, \\
\left\Vert x^{1,j}-x^{2,j}\right\Vert _{1-var,\left[ 0,1\right] } &\leq
&\delta Cl^{\left\vert \nu _{j}\right\vert }.
\end{eqnarray*}

Denote vector fields $V^{i}:=\left( f^{i}\left( \nu _{1}\right) ,\dots
,f^{i}\left( \nu _{K}\right) \right) $, $i=1,2$. For $y_{0}^{i},\widetilde{y}%
_{0}^{i}\in 
%TCIMACRO{\U{211d} }%
%BeginExpansion
\mathbb{R}
%EndExpansion
^{e}$, $i=1,2$, denote $y^{i}:=\pi _{V^{i}}\left( 0,y_{0}^{i};x^{i}\right) $
and $\widetilde{y}^{i}:=\pi _{V^{i}}\left( 0,\widetilde{y}%
_{0}^{i};x^{i}\right) $, $i=1,2$. Then there exists a constant $M=M\left(
C,\gamma ,p,d\right) >0$ such that, when $\lambda l\leq 1$,%
\begin{eqnarray*}
&&\left\vert \left( y_{0,1}^{1}-\widetilde{y}_{0,1}^{1}\right) -\left(
y_{0,1}^{2}-\widetilde{y}_{0,1}^{2}\right) \right\vert \\
&\leq &M\lambda l\left\vert \left( y_{0}^{1}-\widetilde{y}_{0}^{1}\right)
-\left( y_{0}^{2}-\widetilde{y}_{0}^{2}\right) \right\vert \\
&&+M\lambda l\left( \left\vert y_{0}^{1}-\widetilde{y}_{0}^{1}\right\vert
+\left\vert y_{0}^{2}-\widetilde{y}_{0}^{2}\right\vert \right) \left(
\left\vert \widetilde{y}_{0}^{1}-\widetilde{y}_{0}^{2}\right\vert +\delta
+\lambda ^{-1}\left\vert f^{1}-f^{2}\right\vert _{Lip\left( \gamma -1\right)
}\right) \\
&&+M\left( \lambda l\right) ^{\lfloor \gamma \rfloor }\left( \left\vert
y_{0}^{1}-\widetilde{y}_{0}^{1}\right\vert +\left\vert y_{0}^{2}-\widetilde{y%
}_{0}^{2}\right\vert \right) ^{\left\{ \gamma \right\} }\left( \left\vert 
\widetilde{y}_{0}^{1}-\widetilde{y}_{0}^{2}\right\vert +\delta +\lambda
^{-1}\left\vert f^{1}-f^{2}\right\vert _{Lip\left( \gamma -1\right) }\right)
\\
&&+M\lambda l\delta \left\vert y_{0}^{2}-\widetilde{y}_{0}^{2}\right\vert .
\end{eqnarray*}
\end{lemma}

\begin{proof}
Assume $\gamma \in (p,\left[ p\right] +1]$ and denote $N:=\left[ p\right]
=\lfloor \gamma \rfloor $. The constant $M$ in the following proof may
depend on $C,\gamma ,p,d$ and its exact value may change.

Separate the $Lip\left( \gamma -N+1\right) $ term $\left(
d^{N-1}f^{i}\right) \left( f^{i}\right) ^{N-1}$ from $\left\{ f^{i}\left(
\nu _{j}\right) \right\} _{j=1}^{K}$ (if it is one of $f^{i}\left( \nu
_{j}\right) $, $j=1,\dots ,K$, otherwise do nothing). Since $\lambda l\leq 1$%
, $\sum_{j=1}^{K}\left( \lambda l\right) ^{\left\vert \nu _{j}\right\vert
}\leq M\lambda l$. The term associated with $\left( d^{N-1}f^{i}\right)
\left( f^{i}\right) ^{N-1}$ contributes a factor that is comparable to $%
\left( \lambda l\right) ^{N}$. Hence, based on Lemma \ref{Lemma increments
of functions}, by adapting Lemma 10.22 \cite{friz2010multidimensional}, 
\begin{eqnarray*}
&&\left\vert \left( y_{0,t}^{1}-\widetilde{y}_{0,t}^{1}\right) -\left(
y_{0,t}^{2}-\widetilde{y}_{0,t}^{2}\right) \right\vert \\
&\leq &\sum_{j=1}^{K}M\lambda ^{\left\vert \nu _{j}\right\vert
}\int_{0}^{t}\left\vert \left( y_{0,r}^{1}-\widetilde{y}_{0,r}^{1}\right)
-\left( y_{0,r}^{2}-\widetilde{y}_{0,r}^{2}\right) \right\vert \left\vert
dx_{r}^{1,j}\right\vert \\
&&+M\lambda l\left\vert \left( y_{0}^{1}-\widetilde{y}_{0}^{1}\right)
-\left( y_{0}^{2}-\widetilde{y}_{0}^{2}\right) \right\vert \\
&&+M\lambda l\left( \sum_{i=1,2}\left\vert y^{i}-\widetilde{y}%
^{i}\right\vert _{\infty ,\left[ 0,t\right] }\right) \left( \left\vert 
\widetilde{y}^{1}-\widetilde{y}^{2}\right\vert _{\infty ,\left[ 0,t\right]
}+\lambda ^{-1}\left\vert f^{1}-f^{2}\right\vert _{Lip\left( \gamma
-1\right) }\right) \\
&&+M\left( \lambda l\right) ^{N}\left( \sum_{i=1,2}\left\vert y^{i}-%
\widetilde{y}^{i}\right\vert _{\infty ,\left[ 0,t\right] }\right) ^{\left\{
\gamma \right\} }\left( \left\vert \widetilde{y}^{1}-\widetilde{y}%
^{2}\right\vert _{\infty ,\left[ 0,t\right] }+\lambda ^{-1}\left\vert
f^{1}-f^{2}\right\vert _{Lip\left( \gamma -1\right) }\right) \\
&&+M\lambda l\delta \left\vert y^{2}-\widetilde{y}^{2}\right\vert _{\infty ,%
\left[ 0,t\right] }.
\end{eqnarray*}%
Since $\lambda l\leq 1$, based on $\left( \ref{ODE estimate 2}\right) $, we
have%
\begin{equation*}
\left\vert y^{i}-\widetilde{y}^{i}\right\vert _{\infty ,\left[ 0,t\right]
}\leq M\left\vert y_{0}^{i}-\widetilde{y}_{0}^{i}\right\vert \text{, }i=1,2
\end{equation*}%
and%
\begin{equation*}
\left\vert \widetilde{y}^{1}-\widetilde{y}^{2}\right\vert _{\infty ,\left[
0,t\right] }\leq M\left( \left\vert \widetilde{y}_{0}^{1}-\widetilde{y}%
_{0}^{2}\right\vert +\delta +\lambda ^{-1}\left\vert f^{1}-f^{2}\right\vert
_{Lip\left( \gamma -1\right) }\right) .
\end{equation*}%
Based on Gronwall's Lemma and that $\lambda l\leq 1$, the proof is finished.
\end{proof}

Define the symmetry factor $\sigma :\mathcal{F}_{\mathcal{L}}\rightarrow 
%TCIMACRO{\U{2115} }%
%BeginExpansion
\mathbb{N}
%EndExpansion
$ inductively as $\sigma \left( \bullet _{a}\right) :=1$ and 
\begin{equation*}
\sigma \left( \tau _{1}^{n_{1}}\cdots \tau _{k}^{n_{k}}\right) =\sigma
\left( \left[ \tau _{1}^{n_{1}}\cdots \tau _{k}^{n_{k}}\right] _{a}\right)
:=n_{1}!\cdots n_{k}!\sigma \left( \tau _{1}\right) ^{n_{1}}\cdots \sigma
\left( \tau _{k}\right) ^{n_{k}}
\end{equation*}%
where $\tau _{i}\in \mathcal{T}_{\mathcal{L}}$, $i=1,\dots ,k$ are different
labeled trees (labels counted). Based on Proposition 2.3 \cite%
{yang2022remainder}, for a branched rough path $X\in C^{p-var}\left( \left[
0,T\right] ,G_{\mathcal{L}}^{\left[ p\right] }\right) $, if define $\bar{X}:%
\left[ 0,T\right] \rightarrow \left( \mathcal{F}_{\mathcal{L}}^{\left[ p%
\right] }\rightarrow 
%TCIMACRO{\U{211d} }%
%BeginExpansion
\mathbb{R}
%EndExpansion
\right) $ as, for $t\in \left[ 0,T\right] $ and $\tau \in \mathcal{F}_{%
\mathcal{L}}^{\left[ p\right] }$, 
\begin{equation}
\left( \bar{X}_{t},\tau \right) :=\frac{\left( X_{t},\tau \right) }{\sigma
\left( \tau \right) },  \label{equation relation between X and Xbar}
\end{equation}%
then $\bar{X}$ takes values in the step-$\left[ p\right] $ truncated group
of group-like elements in Grossman Larson Hopf algebra (the truncated group
is denoted as $\mathcal{G}_{\mathcal{L}}^{\left[ p\right] }$). Moreover,
based on Proposition 2.3 \cite{yang2022remainder}, for every $0\leq s\leq
t\leq T$ and $\tau \in \mathcal{F}_{\mathcal{L}}^{\left[ p\right] }$,%
\begin{equation}
\left( \bar{X}_{s,t},\tau \right) =\frac{\left( X_{s,t},\tau \right) }{%
\sigma \left( \tau \right) }.  \label{equation relation between X and Xbar 2}
\end{equation}%
We equip $a\in \mathcal{G}_{\mathcal{L}}^{\left[ p\right] }$ with the norm%
\begin{equation*}
\left\Vert a\right\Vert :=\max_{\tau \in \mathcal{F}_{\mathcal{L}}^{\left[ p%
\right] }}\left\vert \left( a,\tau \right) \right\vert ^{\frac{1}{\left\vert
\tau \right\vert }}.
\end{equation*}

\begin{proof}[Proof of Theorem \protect\ref{Theorem main theorem}]
For $i=1,2$, replace $f^{i}$ by $\lambda ^{-1}f^{i}$ and replace $\left(
X_{t}^{i},\tau \right) $ by $\lambda ^{\left\vert \tau \right\vert }\left(
X_{t}^{i},\tau \right) $, $\tau \in \mathcal{F}_{\mathcal{L}}^{\left[ p%
\right] }$. Then the solution to differential equations stays unchanged and $%
\left\vert f^{i}\right\vert _{Lip\left( \gamma \right) }\leq 1$, $i=1,2$.
Suppose $\gamma \in (p,\left[ p\right] +1]$. Denote 
\begin{equation*}
N:=\left[ p\right] \text{ and }\delta :=\rho _{p-\omega ;\left[ 0,T\right]
}\left( X^{1},X^{2}\right) .
\end{equation*}%
The constant $M$ in the following proof may depend on $\gamma ,p,d$ and its
exact value may change.

Firstly suppose $\omega \left( 0,T\right) \leq 1$. For $0\leq s\leq t\leq T$%
, based on $\left( \ref{equation relation between X and Xbar 2}\right) $ and
that $\sigma \left( \tau \right) \geq 1$, we have $\left\Vert \bar{X}%
_{s,t}^{i}\right\Vert \leq \omega \left( s,t\right) ^{\frac{1}{p}}$, $i=1,2$%
, and for $\tau \in \mathcal{F}_{\mathcal{L}}^{\left[ p\right] }$,%
\begin{equation*}
\left\vert \left( \bar{X}_{s,t}^{1}-\bar{X}_{s,t}^{2},\tau \right)
\right\vert \leq \left\vert \left( X_{s,t}^{1}-X_{s,t}^{2},\tau \right)
\right\vert \leq \delta \omega \left( s,t\right) ^{\frac{\left\vert \tau
\right\vert }{p}}\text{.}
\end{equation*}%
Recall $\Phi $ in Notation \ref{Notation isomorphism between Pi rough paths
and branched rough paths} which denotes the isomorphism from a class of $\Pi 
$-rough paths to branched rough paths. Fix $\left[ s,t\right] \subseteq %
\left[ 0,T\right] $. For $\tau \in \mathcal{F}_{\mathcal{L}}^{\left[ p\right]
}$, rescale $\left( \bar{X}_{s,t}^{i},\tau \right) $ by $\omega \left(
s,t\right) ^{-\left\vert \tau \right\vert /p}$ and apply Lemma \ref{Lemma
existence of continuous bounded variation paths for Grossman Larson group}.
Then there exist $x^{i,s,t}=\left( x^{i,s,t,1},\cdots ,x^{i,s,t,K}\right)
\in C^{1-var}\left( \left[ s,t\right] ,%
%TCIMACRO{\U{211d} }%
%BeginExpansion
\mathbb{R}
%EndExpansion
^{K}\right) $, $i=1,2$ such that $\Phi \left( S_{\left[ p\right] }\left(
x^{i,s,t}\right) \right) =\bar{X}_{s,t}^{i}$, $i=1,2$\ and for $j=1,\dots ,K$%
,%
\begin{eqnarray}
\max_{i=1,2}\left\Vert x^{i,s,t,j}\right\Vert _{1-var,\left[ s,t\right] }
&\leq &M\omega \left( s,t\right) ^{\frac{\left\vert \nu _{j}\right\vert }{p}}%
\text{, }i=1,2  \label{inner estimate 2} \\
\left\Vert x^{1,s,t,j}-x^{2,s,t,j}\right\Vert _{1-var,\left[ s,t\right] }
&\leq &\delta M\omega \left( s,t\right) ^{\frac{\left\vert \nu
_{j}\right\vert }{p}}.  \label{inner estimate 3}
\end{eqnarray}%
Let $y^{i,s,t}:\left[ s,t\right] \rightarrow 
%TCIMACRO{\U{211d} }%
%BeginExpansion
\mathbb{R}
%EndExpansion
^{e}$ denote the unique solution of the ODE%
\begin{equation*}
dy_{r}^{i,s,t}=\sum_{j=1}^{K}f^{i}\left( \nu _{j}\right) \left(
y_{r}^{i,s,t}\right) dx_{r}^{i,s,t,j},\text{ }y_{s}^{i,s,t}=y_{s}^{i}.
\end{equation*}%
Denote%
\begin{eqnarray*}
\Gamma _{s,t}^{i} &:&=y_{s,t}^{i}-y_{s,t}^{i,s,t}\text{, }i=1,2 \\
\bar{\Gamma}_{s,t} &:&=\Gamma _{s,t}^{1}-\Gamma _{s,t}^{2}.
\end{eqnarray*}%
Since we assumed $\omega \left( 0,T\right) \leq 1$, by setting $\omega
\left( 0,T\right) =1$ in Proposition 3.17 in \cite{yang2022remainder}, we
have%
\begin{eqnarray}
\left\vert \Gamma _{s,t}^{i}\right\vert &\leq &M\omega \left( s,t\right) ^{%
\frac{\left[ p\right] +1}{p}}\text{, }i=1,2  \label{inner estimate 12} \\
\left\vert \bar{\Gamma}_{s,t}\right\vert &\leq &M\omega \left( s,t\right) ^{%
\frac{\left[ p\right] +1}{p}}.  \notag
\end{eqnarray}%
In fact, based on the construction, $x^{i,s,t}\in C^{1-var}\left( \left[ s,t%
\right] ,%
%TCIMACRO{\U{211d} }%
%BeginExpansion
\mathbb{R}
%EndExpansion
^{K}\right) $ here may not be a geodesic associated with $\bar{X}_{s,t}^{i}$
in the sense of Definition 3.2 \cite{yang2022remainder}. The estimate of
Proposition 3.17 \cite{yang2022remainder} applies, because $\Phi \left( S_{%
\left[ p\right] }\left( x^{i,s,t}\right) _{s,t}\right) =\bar{X}_{s,t}^{i}$
and for $j=1,\dots ,K$, $\left\Vert x^{i,s,t,j}\right\Vert _{1-var,\left[ s,t%
\right] }\leq M\omega \left( s,t\right) ^{\left\vert \nu _{j}\right\vert /p}$
based on Lemma \ref{Lemma existence of continuous bounded variation paths
for Grossman Larson group}.

For $i=1,2$ and $0\leq s\leq t\leq u\leq T$, let $x^{i,s,t,u}\in
C^{1-var}\left( \left[ s,u\right] ,%
%TCIMACRO{\U{211d} }%
%BeginExpansion
\mathbb{R}
%EndExpansion
^{K}\right) $ denote the concatenation of $x^{i,s,t}$ with $x^{i,t,u}$.
Denote by $y^{i,s,t,u}:\left[ s,u\right] \rightarrow 
%TCIMACRO{\U{211d} }%
%BeginExpansion
\mathbb{R}
%EndExpansion
^{e}$ the solution of the ODE%
\begin{equation*}
dy_{r}^{i,s,t,u}=\sum_{j=1}^{K}f^{i}\left( \nu _{j}\right) \left(
y_{r}^{i,s,t,u}\right) dx_{r}^{i,s,t,u,j},\text{ }y_{s}^{i,s,t,u}=y_{s}^{i}.
\end{equation*}%
For $i=1,2$, denote%
\begin{equation*}
A^{i}:=y_{s,u}^{i,s,t,u}-y_{s,u}^{i,s,u}\text{, }%
B^{i}:=y_{t}^{i,s,t}+y_{t,u}^{i,t,u}-y_{u}^{i,s,t,u}
\end{equation*}%
and denote%
\begin{equation*}
\bar{A}:=A^{1}-A^{2}\text{, }\bar{B}=B^{1}-B^{2}
\end{equation*}%
so that%
\begin{equation*}
\bar{\Gamma}_{s,u}-\bar{\Gamma}_{s,t}-\bar{\Gamma}_{t,u}=\bar{A}+\bar{B}.
\end{equation*}

Denote 
\begin{equation*}
\bar{\delta}:=\delta +\lambda ^{-1}\left\vert f^{1}-f^{2}\right\vert
_{Lip\left( \gamma -1\right) }.
\end{equation*}%
As $S_{\left[ p\right] }\left( x^{i,s,t,u}\right) _{s,u}=S_{\left[ p\right]
}\left( x^{i,s,u}\right) _{s,u}$, $i=1,2$, based on $\left( \ref{inner
estimate 2}\right) $ and $\left( \ref{inner estimate 3}\right) $, apply
Lemma \ref{Lemma Abar},%
\begin{equation}
\left\vert \bar{A}\right\vert \leq M\omega \left( s,u\right) ^{\frac{\gamma 
}{p}}\left( \left\vert y_{s}^{1}-y_{s}^{2}\right\vert +\bar{\delta}\right) .
\label{inner estimate 4}
\end{equation}%
Denote vector fields $V^{i}:=\left( f^{i}\left( \nu _{1}\right) ,\cdots
,f^{i}\left( \nu _{K}\right) \right) $, $i=1,2$. Based on Lemma \ref{Lemma
Bbar},%
\begin{eqnarray}
\left\vert \bar{B}\right\vert &=&\left\vert \left( \pi _{V^{1}}\left(
t,y_{t}^{1};x^{1,t,u}\right) _{t,u}-\pi _{V^{1}}\left( t,y_{t}^{1}-\Gamma
_{s,t}^{1};x^{1,t,u}\right) _{t,u}\right) \right.  \label{inner estimate 5}
\\
&&\left. -\left( \pi _{V^{2}}\left( t,y_{t}^{2};x^{2,t,u}\right) _{t,u}-\pi
_{V^{2}}\left( t,y_{t}^{2}-\Gamma _{s,t}^{2};x^{2,t,u}\right) _{t,u}\right)
\right\vert  \notag \\
&\leq &M\omega \left( s,u\right) ^{1/p}\left\vert \bar{\Gamma}%
_{s,t}\right\vert  \notag \\
&&+M\left( \omega \left( s,u\right) ^{1/p}\left( \left\vert \Gamma
_{s,t}^{1}\right\vert +\left\vert \Gamma _{s,t}^{2}\right\vert \right)
+\omega \left( s,u\right) ^{N/p}\left( \left\vert \Gamma
_{s,t}^{1}\right\vert +\left\vert \Gamma _{s,t}^{2}\right\vert \right)
^{\left\{ \gamma \right\} }\right)  \notag \\
&&\times \left( \left\vert y_{t}^{1}-y_{t}^{2}\right\vert +\bar{\delta}%
\right)  \notag \\
&&+M\omega \left( s,u\right) ^{1/p}\delta \left\vert \Gamma
_{s,t}^{2}\right\vert  \notag
\end{eqnarray}%
Based on $\left( \ref{inner estimate 12}\right) $, $\left\vert \Gamma
_{s,t}^{i}\right\vert \leq M\omega \left( s,t\right) ^{\frac{\left[ p\right]
+1}{p}}$, $i=1,2$. As $\omega \left( 0,T\right) \leq 1$, combine $\left( \ref%
{inner estimate 4}\right) $ and $\left( \ref{inner estimate 5}\right) $,%
\begin{eqnarray}
\left\vert \bar{\Gamma}_{s,u}\right\vert &\leq &\left\vert \bar{A}%
\right\vert +\left\vert \bar{B}\right\vert +\left\vert \bar{\Gamma}%
_{s,t}\right\vert +\left\vert \overline{\Gamma }_{t,u}\right\vert
\label{inner estimate 6} \\
&\leq &\exp \left( M\omega \left( s,u\right) ^{1/p}\right) \left( \left\vert 
\bar{\Gamma}_{s,t}\right\vert +\left\vert \overline{\Gamma }%
_{t,u}\right\vert \right)  \notag \\
&&+M\omega \left( s,u\right) ^{\gamma /p}\left( \sup_{r\in \left[ s,u\right]
}\left\vert y_{r}^{1}-y_{r}^{2}\right\vert +\bar{\delta}\right) .  \notag
\end{eqnarray}%
Since $\left\vert f^{1}\left( \nu _{j}\right) -f^{2}\left( \nu _{j}\right)
\right\vert _{\infty }\leq M\lambda ^{\left\vert \nu _{j}\right\vert
-1}\left\vert f^{1}-f^{2}\right\vert _{Lip\left( \gamma -1\right) }$ for $%
j=1,\dots ,K$ and $\omega \left( 0,T\right) \leq 1$, based on $\left( \ref%
{ODE estimate 1}\right) $,%
\begin{equation}
\left\vert y_{s,t}^{1,s,t}-y_{s,t}^{2,s,t}\right\vert \leq M\left(
\left\vert y_{s}^{1}-y_{s}^{2}\right\vert +\bar{\delta}\right) \omega \left(
s,t\right) ^{1/p}.  \label{inner estimate 7}
\end{equation}%
Combine $\left( \ref{inner estimate 6}\right) $, $\left( \ref{inner estimate
7}\right) $ and that $\left\vert \bar{\Gamma}_{s,t}\right\vert \leq M\omega
\left( s,t\right) ^{\frac{\left[ p\right] +1}{p}}$, based on Proposition
10.63 \cite{friz2010multidimensional} (applying to the interval $\left[ s,t%
\right] $), we have%
\begin{equation*}
\left\vert \bar{\Gamma}_{s,t}\right\vert \leq M\left( \left\vert
y_{s}^{1}-y_{s}^{2}\right\vert +\bar{\delta}\right) \omega \left( s,t\right)
^{\gamma /p}\exp \left( M\omega \left( s,t\right) \right) .
\end{equation*}%
Hence, when $\omega \left( s,t\right) \leq 1$, 
\begin{eqnarray}
\left\vert y_{s,t}^{1}-y_{s,t}^{2}\right\vert &\leq &\left\vert
y_{s,t}^{1,s,t}-y_{s,t}^{2,s,t}\right\vert +\left\vert \bar{\Gamma}%
_{s,t}\right\vert  \label{inner estimate 8} \\
&\leq &M\left( \left\vert y_{s}^{1}-y_{s}^{2}\right\vert +\bar{\delta}%
\right) \omega \left( s,t\right) ^{1/p}\exp \left( M\omega \left( s,t\right)
\right) .  \notag
\end{eqnarray}

Suppose $\omega \left( 0,T\right) >1$. When $\omega \left( s,t\right) \leq 1$%
, the estimates above apply. When $\omega \left( s,t\right) >1$, divide $%
\left[ s,t\right] =\cup _{i=0}^{n-1}\left[ t_{i},t_{i+1}\right] $ such that $%
\omega \left( t_{i},t_{i+1}\right) =1$, $i=0,\dots ,n-2$ and $\omega \left(
t_{n-1},t_{n}\right) \leq 1$. By the super-additivity of $\omega $ (i.e. $%
\omega \left( s,t\right) +\omega \left( t,u\right) \leq \omega \left(
s,u\right) $ for $s\leq t\leq u$),%
\begin{equation}
n=\sum_{i=0}^{n-2}\omega \left( t_{i},t_{i+1}\right) +1\leq \omega \left(
s,t\right) +1\leq 2\omega \left( s,t\right) \text{.}
\label{inner estimate 9}
\end{equation}%
Since $\omega \left( t_{i},t_{i+1}\right) \leq 1$, $i=0,\dots ,n-1$, based
on $\left( \ref{inner estimate 8}\right) $, there exists $M_{0}>0$ such that%
\begin{equation*}
\left\vert y_{t_{i},t_{i+1}}^{1}-y_{t_{i},t_{i+1}}^{2}\right\vert \leq
M_{0}\left( \left\vert y_{t_{i}}^{1}-y_{t_{i}}^{2}\right\vert +\bar{\delta}%
\right)
\end{equation*}%
and%
\begin{eqnarray*}
\left\vert y_{t_{i}}^{1}-y_{t_{i}}^{2}\right\vert &\leq &\left\vert
y_{t_{i-1}}^{1}-y_{t_{i-1}}^{2}\right\vert +\left\vert
y_{t_{i-1},t_{i}}^{1}-y_{t_{i-1},t_{i}}^{2}\right\vert \\
&\leq &\left( 1+M_{0}\right) \left\vert
y_{t_{i-1}}^{1}-y_{t_{i-1}}^{2}\right\vert +M_{0}\bar{\delta} \\
&\leq &\left( 1+M_{0}\right) ^{i}\left\vert y_{s}^{1}-y_{s}^{2}\right\vert
+M_{0}\left( \sum_{j=0}^{i-1}\left( 1+M_{0}\right) ^{j}\right) \bar{\delta}.
\end{eqnarray*}%
Hence%
\begin{equation*}
\left\vert y_{t_{i},t_{i+1}}^{1}-y_{t_{i},t_{i+1}}^{2}\right\vert \leq
M_{0}\left( 1+M_{0}\right) ^{i}\left( \left\vert
y_{s}^{1}-y_{s}^{2}\right\vert +\bar{\delta}\right)
\end{equation*}%
and%
\begin{eqnarray}
&&\left\vert y_{s,t}^{1}-y_{s,t}^{2}\right\vert  \label{inner estimate 10} \\
&\leq &\sum_{i=0}^{n-1}\left\vert
y_{t_{i},t_{i+1}}^{1}-y_{t_{i},t_{i+1}}^{2}\right\vert  \notag \\
&\leq &\sum_{i=0}^{n-1}M_{0}\left( 1+M_{0}\right) ^{i}\left( \left\vert
y_{s}^{1}-y_{s}^{2}\right\vert +\bar{\delta}\right)  \notag \\
&\leq &\left( 1+M_{0}\right) ^{n}\left( \left\vert
y_{s}^{1}-y_{s}^{2}\right\vert +\bar{\delta}\right)  \notag \\
&=&\exp \left( n\ln \left( 1+M_{0}\right) \right) \left( \left\vert
y_{s}^{1}-y_{s}^{2}\right\vert +\bar{\delta}\right)  \notag \\
&\leq &\left( \left\vert y_{s}^{1}-y_{s}^{2}\right\vert +\bar{\delta}\right)
\exp \left( M\omega \left( s,t\right) \right)  \notag
\end{eqnarray}%
where in the last step we used $\left( \ref{inner estimate 9}\right) $. In
particular, when $\left[ s,t\right] =\left[ 0,s\right] $,%
\begin{eqnarray}
\left\vert y_{s}^{1}-y_{s}^{2}\right\vert &\leq &\left\vert
y_{0}^{1}-y_{0}^{2}\right\vert +\left\vert y_{0,s}^{1}-y_{0,s}^{2}\right\vert
\label{inner estimate 11} \\
&\leq &2\left( \left\vert y_{0}^{1}-y_{0}^{2}\right\vert +\bar{\delta}%
\right) \exp \left( M\omega \left( 0,s\right) \right) .  \notag
\end{eqnarray}%
Combining $\left( \ref{inner estimate 10}\right) $, $\left( \ref{inner
estimate 11}\right) $ and that $\omega \left( s,t\right) \geq 1$, we have 
\begin{equation*}
\left\vert y_{s,t}^{1}-y_{s,t}^{2}\right\vert \leq M\left( \left\vert
y_{0}^{1}-y_{0}^{2}\right\vert +\bar{\delta}\right) \omega \left( s,t\right)
^{1/p}\exp \left( M\omega \left( 0,t\right) \right) .
\end{equation*}%
Combining $\left( \ref{inner estimate 8}\right) $, $\left( \ref{inner
estimate 11}\right) $ and the super-additivity of $\omega $, the same result
holds when $\omega \left( s,t\right) \leq 1$. Then the proposed estimate
holds as $\bar{\delta}:=\rho _{p-\omega ,\left[ 0,T\right] }\left(
X^{1},X^{2}\right) +\lambda ^{-1}\left\vert f^{1}-f^{2}\right\vert
_{Lip\left( \gamma -1\right) }$.
\end{proof}

\bibliographystyle{unsrt}
\bibliography{acompat,roughpath}

\end{document}